\documentclass[12pt]{article}
\usepackage[top=2.5cm,bottom=2.5cm,left=2.5cm,right=2.5cm]{geometry}
\usepackage{amssymb}
\usepackage{amsmath,amsthm}
\usepackage[latin1]{inputenc}
\usepackage[dvips]{graphicx}
\usepackage{hyperref}
\usepackage{color}
\usepackage{mathrsfs}
\usepackage{enumerate}
\usepackage{tikz}
\usepackage{xifthen}
\usepackage{verbatim}

\hypersetup{colorlinks=true}

\hypersetup{colorlinks=true, linkcolor=blue, citecolor=blue,urlcolor=blue}

\newtheorem{remark}{Remark}[section]

\newtheorem{lemma}[remark]{Lemma}
\newtheorem{theorem}[remark]{Theorem}
\newtheorem{proposition}[remark]{Proposition}

\newtheorem{corollary}[remark]{Corollary}

\title{Italian domination in rooted product graphs}

\author{R. Hern\'andez-Ortiz$^{1}$, L. P. Montejano$^2$, J. A. Rodr{\'\i}guez-Vel\'azquez$^{2}$
\\
\\
$^1${\small Departamento de Matem\' aticas Aplicadas y Sistemas}\\{\small Universidad Aut\'onoma Metropolitana, Ciudad de M\' exico, Mexico}
\\
$^2${\small 
Universitat Rovira i Virgili, Departament d'Enginyeria
Inform\`{a}tica i Matem\`{a}tiques}\\
{\small
Av. Pa\"{i}sos Catalans 26, 43007
Tarragona, Spain}
\\{\small Email addresses: rangel@ciencias.unam.mx, luispedro.montejano@urv.cat, juanalberto.rodriguez@urv.cat}
}

\begin{document}
\maketitle 

\begin{abstract}
In this article, we obtain closed formulae for the Italian domination number of rooted  product graphs.
As a particular case of the study, we derive the corresponding formulas for corona graphs, and we provide an alternative proof that the 
problem of computing the Italian domination number of a graph is NP-hard.
\end{abstract}

 

\section{Introduction}
Consider the following approach to protecting a network. 
Suppose that one or more entities are stationed at some of the nodes of a network and that an entity at a node $v$ can deal with a problem produced in $v$ or in its neighbouring nodes. 
Depending on the nature of the network, an entity could consist of a robot, an observer, a spy, an intruder, a legion,  a guard,  and so on. Informally, we say that a network (or its underlying graph) is protected under a  placement of entities if there exists at least one entity available to handle a problem at any node.

Let $G$ be a simple graph whose vertex set is $V(G)$ and whose edge set is $E(G)$. Consider a function $f: V(G)\longrightarrow \{0,1,2\}$ where $f(v)$ denotes the number of entities stationed at vertex $v$. For every   $i\in \{0,1,2\}$ we define the sets $V_i=\{v\in V(G):\; f(v)=i\}$. We will identify the function $f$ with the partition of the vertex set induced by $f$ and, with this end,  we will write $f(V_0,V_1, V_2).$ The weight of $f$ is defined to be $$\omega(f)=f(V(G))=\sum_{v\in V(G)}f(v)=\sum_ii|V_i|.$$ 
  We now consider two particular strategies of graph protection; the so-called Roman domination and the so-called Italian domination. As we can expect, the minimum number of entities required for protection under each strategy is of interest.

Let $N(v)$ be the open neighbourhood of $v\in V(G)$. A function $f(V_0,V_1,V_2)$ is a \textit{Roman dominating function} (RDF) if  $N(v)\cap V_2\ne \emptyset$ for every vertex $v\in V_0$. The \textit{Roman domination number}, denoted by  $\gamma_R(G)$, is defined to be $$\gamma_R(G)=\min\{w(f):\, f \text{ is a RDF on } G\}.$$  
This strategy of graph protection was formally proposed by Cockayne et al.\ in \cite{Cockayne2004}. 
For simplicity, a Roman dominating function with minimum weight $\gamma_{_R}(G)$ on $G$ will be called a $\gamma_{_R}(G)$-\emph{function}. 

A generalization of Roman domination called Italian domination was introduced by Chella\-li et al.\ in \cite{CHELLALI201622}, where it was called Roman $\{2\}$-domination. The concept was stu\-died further in \cite{HENNING2017557,Klostermeyer201920}. An \emph{Italian dominating function} (IDF) on a graph $G$   is a
function $f(V_0,V_1,V_2)$ satisfying that $f(N(v))=\sum_{u\in N(v)}f(u)\ge 2$ for every $v\in V_0$, i.e., $f(V_0,V_1,V_2)$ is an IDF if $N(v)\cap V_2\ne \emptyset$ or $|N(v)\cap V_1|\ge 2$ for every $v\in V_0$.

The \textit{Italian domination number}, denoted by  $\gamma_{_I}(G)$, is is defined to be $$\gamma_{_I}(G)=\min\{w(f):\, f \text{ is an IDF on } G\}.$$  An Italian dominating function with minimum weight $\gamma_{_I}(G)$ on $G$ will be called a $\gamma_{_I}(G)$-\emph{function}.
We will assume a similar agreement when referring to the optimal functions (and sets) associated with other parameters defined below.

Since the problem of computing $\gamma_I(G)$ is NP-hard \cite{CHELLALI201622}, the need to obtain formulas for this parameter arises. In this article,   we address this problem for the case of rooted  product graphs and corona product graphs.

Given a graph $G$ and a graph $H$ with root vertex $v\in V(H)$, the \emph{rooted product graph} $G\circ_v H$ is defined to be the graph obtained from $G$ and $H$ by taking one copy of $G$ and $|V(G)|$ copies of $H$ and identifying the  $i^{th}$ vertex of $G$ with the root $v$ in the $i^{th}$ copy of $H$ for each  $i\in \{1,\dots , |V(G)|\}$. For every vertex $x\in V(G)$,  the copy of $H$ in $G\circ_v H$ containing $x$ will be denoted by $H_x$, and for every  IDF  $f$ on $G\circ_v H$, the restriction of $f$  to $V(H_x)$ and $V(H_x)\setminus \{x\}$
 will be denoted by $f_x$ and $f_x^-$,  respectively.
  Notice that $V(G\circ_v H)=\cup_{x\in V(G)}V(H_x)$ and so, if $f$ is a $\gamma_{_I}(G\circ_v H)$-function, then  $$\gamma_{_I}(G\circ_v H)=\sum_{x\in V(G)}\omega(f_x)=\sum_{x\in V(G)}\omega(f_x^-)+\sum_{x\in V(G)}f(x).$$

Throughout the paper, we will use the notation $K_t$,  $C_t$  and $P_t$ for complete graphs,  cycle graphs and path graphs of order $t$, respectively. 
We will use the notation  $G \cong H$ if $G$ and $H$ are isomorphic graphs.

For the remainder of the paper, definitions will be introduced whenever a concept is needed.

\section{Italian domination of rooted product graphs}

To begin the study we need to establish some preliminary tools.

\begin{lemma}\label{restriction-H}
If $f(V_0,V_1,V_2)$ is a $\gamma_{_I}(G\circ_v H)$-function and $x\in V(G)$, then $\omega(f_x)\geq \gamma_{_I}(H)-1$. Furthermore, if $\omega(f_x)=\gamma_{_I}(H)-1$, then $f(x)=0$.
\end{lemma}
\begin{proof}
Suppose to the contrary that there exists a vertex $x\in V(G)$ such that  $\omega(f_x)\leq\gamma_{_I}(H)-2$. Now, if $f(x)>0$, then $f_x$ is an IDF on $H_x$ and $\omega(f_x)< \gamma_I(H_x)$, which is a contradiction; while if $f(x)=0$, then the   function $g$, defined  by $g(x)=1$ and  $g(v)=f_x(v)$ whenever $v\neq x$, is an IDF on $H_x$ of weight  $\omega(g)=\omega(f_x)+1 < \gamma_{_I}(H_x)$, which is a contradiction again. Hence,  $\omega(f_x)\geq \gamma_{_I}(H)-1$ for every  $x\in V(G)$.

Now, if there exists a vertex  $x\in V(G)$ such that $\omega(f_x)=\gamma_{_I}(H)-1$ and $f(x)>0$, then $f_x$ is an IDF on $H_x$ of  weight $\omega(f_x)<\gamma_{_I}(H_x)$, which is a contradiction. Therefore, if $ \omega(f_x)=\gamma_{_I}(H)-1$, then $f(x)=0$.
\end{proof}

For every $\gamma_{_I}(G\circ_v H)$-function $f(V_0,V_1,V_2)$   we define the   sets
  $$\mathcal{A}_f=\{x\in V(G): \omega(f_x)\geq \gamma_{_I}(H)\}$$ and  $$\mathcal{B}_f=\{x\in V(G) : \omega(f_x)=\gamma_{_I}(H)-1\}.$$ Notice that by Lemma \ref{restriction-H} we have that if $\mathcal{B}_f\ne \emptyset $, then    $\{\mathcal{A}_f, \mathcal{B}_f\}$ is a partition of the vertex set of $G$ and so 
  $$\gamma_{_I}(G\circ_v H)=\sum_{x\in \mathcal{A}_f}\omega(f_x)+\sum_{x\in    \mathcal{B}_f}\omega(f_x).$$
  
  The following consequence of Lemma \ref{restriction-H} is immediate. 

\begin{corollary}\label{Corollary-of-Lemma1}
Let $f$ be a $\gamma_{_I}(G\circ_v H)$-function.  If $\mathcal{B}_f\neq \emptyset$, then either $H\in \{K_1,K_2,\overline{K_2}\}$  or $\gamma_{_I}(H)\ge 3$.
\end{corollary}

\begin{lemma}\label{CorollaryA_fDominante}
If $f$ is a  $\gamma_{_I}(G\circ_v H)$-function, then   $\mathcal{A}_f$ is a dominating set of~$G$. 
  \end{lemma}

\begin{proof}Let $f(V_0,V_1,V_2)$ be  a  $\gamma_{_I}(G\circ_v H)$-function. Notice that  Lemma \ref{restriction-H} leads to  $\mathcal{B}_f\subseteq V_0$.  Now, since $f$ is a $\gamma_{_I}(G\circ_v H)$-function,  if there exists  $x\in \mathcal{B}_f $ such that $N(x)\cap V(G)\cap  (V_1\cup V_2)=\emptyset$, then
 $f_x$  is an IDF on $H_x$ of weight $\omega(f_x)=\gamma_{_I}(H)-1<\gamma_{_I}(H_x)$,  which is a contradiction. Hence, every vertex
 $x\in \mathcal{B}_f $ is adjacent to some vertex belonging to $V(G)\cap  (V_1\cup V_2)\subseteq \mathcal{A}_f\setminus V_0$. Therefore, $\mathcal{A}_f$ is a dominating set of~$G$. 
\end{proof}

  \begin{lemma}\label{lemmaA}
If $f(V_0,V_1,V_2)$ is a $\gamma_{_I}(G\circ_v H)$-function such  that $\mathcal{B}_f\neq \emptyset$, then  $\omega(f_x)=\gamma_{_I}(H)$ for every  $x\in \mathcal{A}_f\cap (V_0\cup V_1)$; while 
   $\omega(f_x)\leq \gamma_{_I}(H)+1$  for every $x\in \mathcal{A}_f\cap V_2$.
\end{lemma}

  \begin{proof}
Let $f$ be a $\gamma_{_I}(G\circ_v H)$-function and $u\in V(G)$ such that $u\in \mathcal{B}_f$. 
First, suppose to the contrary that there exists $x\in \mathcal{A}_f\cap (V_0\cup V_1)$ such that $\omega(f_x)\geq \gamma_{_I}(H)+1$. Let $g$ be the  function  on $G\circ_vH$ defined by  $g(w)=f(w)$ for every $w\notin V(H_x)$, $g(x)=1$ and  $g_x^-$  is induced by  $f_u^-$. It is readily seen that $g$ is an IDF on $G\circ_vH$ and  $\omega(g)\le \omega(f)-1=\gamma_{_I}(G\circ_v H)-1$, which is  a contradiction. Therefore, $\omega(f_x)=\gamma_{_I}(H)$ for every $x\in \mathcal{A}_f\cap (V_0\cup V_1)$.

Now, suppose to the contrary that there exists $x\in \mathcal{A}_f\cap V_2$ such that $\omega(f_x)\geq \gamma_{_I}(H)+2$.  In this case we define a function $g$ on $G\circ_vH$ by  $g(w)=f(w)$ for every $w\notin V(H_x)$, $g(x)=2$  and  $g_x^-$  is induced by  $f_u^-$. It is readily seen that $g$ is an IDF on $G\circ_vH$ and  $\omega(G)\le \omega(f)-1=\gamma_{_I}(G\circ_v H)-1$, which is  a contradiction. Therefore, $\omega(f_x)\leq\gamma_{_I}(H)+1$ for every $x\in \mathcal{A}_f\cap V_2$.
\end{proof}

  Let us define the sets $$\mathcal{A}_f^{i,j}=\{x\in \mathcal{A}_f: f(x)=i \text{ and } \omega(f_x)=j\},$$  where $i\in \{0,1,2\}$, $j\in \{\gamma_{_I}(H), \gamma_{_I}(H)+1\}$. For simplicity,  we will  use the notation  $m=\gamma_{_I}(H)$ in some lemmas and proofs, specially when $\gamma_{_I}(H)$ is a superscript.
  
  From  Lemma \ref{lemmaA}  we have the following consequence. 
  
  \begin{corollary}\label{CorollaryLemma3}
  If $f(V_0,V_1,V_2)$ is a $\gamma_{_I}(G\circ_v H)$-function such  that $\mathcal{B}_f\neq \emptyset$, then $$\mathcal{A}_f=\mathcal{A}_f^{0,m}\cup \mathcal{A}_f^{1,m}\cup \mathcal{A}_f^{2,m}\cup \mathcal{A}_f^{2,m+1}.$$
\end{corollary}

\begin{lemma}\label{existenciaf}
Let $f$ be a $\gamma_{_I}(G\circ_v H)$-function.  If $\mathcal{B}_f\neq \emptyset$, then  there exists a $\gamma_{_I}(G\circ_v H)$-function $g$ such that $\mathcal{B}_g=\mathcal{B}_f$ and $$\mathcal{A}_g\in \{\mathcal{A}_g^{1,m},\mathcal{A}_g^{2,m}, \mathcal{A}_g^{2,m+1},\mathcal{A}_g^{1,m}\cup\mathcal{A}_g^{2,m+1}\}.$$
\end{lemma}
  \begin{proof}
  Let $f$ be a $\gamma_{_I}(G\circ_v H)$-function with $\mathcal{B}_f\neq \emptyset$. Notice that, by Lemma \ref{CorollaryA_fDominante},  $\mathcal{A}_f\neq\emptyset$.
Now, since $f$ is a  $\gamma_{_I}(G\circ_v H)$-function, if $\mathcal{A}_f^{2,m}\neq\emptyset$, then $\mathcal{A}_f^{2,m+1}=\emptyset$.  Furthermore,  if $\mathcal{A}_f^{1,m}\neq\emptyset$ and $\mathcal{A}_f^{0,m}\neq\emptyset$, then we fix $y\in\mathcal{A}_f^{1,m}$ and we  define a $\gamma_{_I}(G\circ_v H)$-function $g$  such that for every $x\in \mathcal{A}_f^{0,m}$,  $g_x$ is induced by $f_y$ and $g_z=f_z$ for every $z\in V(G)\setminus \mathcal{A}_f^{0,m}$. In such a case, $\mathcal{A}_g^{1,m}\neq\emptyset$ and $\mathcal{A}_g^{0,m}=\emptyset$.

 Using similar arguments we can show that  if $\mathcal{A}_f^{2,m}\neq\emptyset$, then there exists a  $\gamma_{_I}(G\circ_v H)$-function $g$ such that $\mathcal{A}_g^{0,m}\cup \mathcal{A}_g^{1,m}\cup \mathcal{A}_g^{2,m+1}=\emptyset$.
 
Hence,  by Corollary \ref{CorollaryLemma3} we conclude that $$\mathcal{A}_g\in \{\mathcal{A}_g^{0,m},\mathcal{A}_g^{1,m},\mathcal{A}_g^{2,m},\mathcal{A}_g^{0,m}\cup \mathcal{A}_g^{2,m+1},\mathcal{A}_g^{1,m}\cup\mathcal{A}_g^{2,m+1}\}.$$
Finally, if $\mathcal{A}_g^{0,m}\ne \emptyset$, then we fix $y\in \mathcal{B}_g$ and 
we define a function $h$ on $G\circ_v H$ by  $h_z=g_z$ for every $z\in V(G)\setminus \mathcal{A}_g^{0,m}$ and for every $x\in \mathcal{A}_g^{0,m}$ we set $h(x)=1$ and $h_x^-$ is induced by $g^-_y$ . Notice that $h$ is an IDF of weight $\omega(h)=\omega(g)=\omega(f)$ and $\mathcal{A}_h\in \{ \mathcal{A}_h^{1,m},\mathcal{A}_h^{2,m},  \mathcal{A}_h^{2,m+1},\mathcal{A}_h^{1,m}\cup\mathcal{A}_h^{2,m+1}\}.$
Therefore, the result follows.
\end{proof}

\begin{proposition}\label{lemma-bound-B}
If there exists a $\gamma_{_I}(G\circ_v H)$-function $f$ such that $\mathcal{B}_f\neq\emptyset$, then $$\gamma_{_I}(G\circ_v H)\leq n(G)(\gamma_{_I}(H)-1)+\gamma_{_I}(G).$$
\end{proposition}
\begin{proof}
Let $f$ be a $\gamma_{_I}(G\circ_v H)$-function and $u\in V(G)$ such that $u\in \mathcal{B}_f$. Let $h$ be a  $\gamma_{_I}(G)$-function.
By Lemma \ref{restriction-H}, $f(u)=0$, so that  $f_u^-$ is an IDF on $H_u-\{u\}$. Notice that $\omega(f_u^-)=\omega(f_u)=\gamma_{_I}(H)-1$. Consider the function $g$  on $G\circ_v H$ such that  $g_x^-$ is induced by $f_u^-$ and $g(x)=h(x)$ for every vertex $x\in V(G)$. Thus, $g$ is an IDF on $G\circ_v H$ of weight $\omega(g)=n(G)\omega(f_u^-)+\omega(h)=n(G)(\gamma_{_I}(H)-1)+\gamma_{_I}(G)$, concluding that $\gamma_{_I}(G\circ_v H)\leq n(G)(\gamma_{_I}(H)-1)+\gamma_{_I}(G)$.
\end{proof}

\begin{theorem}[Trichotomy]\label{teo-bounds-rooted}

For any graph $G$,  any graph $H$ and any  vertex $v\in~V(H)$, 
\begin{itemize}
\item $\gamma_{_I}(G\circ_v H)=n(G)(\gamma_{_I}(H)-1)+\gamma(G)$ or
\item $\gamma_{_I}(G\circ_v H)= n(G)(\gamma_{_I}(H)-1)+\gamma_{_I}(G)$ or

\item $\gamma_{_I}(G\circ_v H)= n(G)\gamma_{_I}(H).$
\end{itemize}
Furthermore, the following statements hold for any pair of  $\gamma_{_I}(G\circ_v H)$-functions $f$ and $f'$.
\begin{itemize}
\item  $\mathcal{B}_f=~\emptyset$ if and only if $\mathcal{B}_{f'}=~\emptyset$.
\item $\gamma_{_I}(G\circ_v H)=n(G)\gamma_{_I}(H)$ if and only if $\mathcal{B}_f=~\emptyset$. 
\end{itemize}
\end{theorem}

\begin{proof} Let $f(V_0,V_1,V_2)$ be a $\gamma_{_I}(G\circ_v H)$-function.
If $\mathcal{B}_f=~\emptyset$, then  $\omega(f_x)\ge \gamma_{_I}(H)$ for every $x\in V(G)$, which implies that $\gamma_{_I}(G\circ_v H)\ge n(G)\gamma_{_I}(H)$. Hence,  $\gamma_{_I}(G\circ_v H)=n(G)\gamma_{_I}(H)$, as we always can construct an IDF $g$ such that  $\omega(g_x)=\gamma_{_I}(H)$ for every $x\in V(G)$. 

From now on we consider the case $\mathcal{B}_f\neq\emptyset$, and so we can assume that $f$ is a $\gamma_{_I}(G\circ_v H)$-function  which satisfies  Lemma \ref{existenciaf}. 

First, suppose that there exists $x\in \mathcal{B}_f$ such that $f(y)>0$ for some $y\in N(x)\cap V(H_x)$. Let $S$ be a $\gamma(G)$-set and consider the function $g$ on $G\circ_v H$ where  $g_u^-$ is induced by $f_x^-$  for every $u\in V(G)$, $g(u)=1$ for every $u\in S$ and $g(u)=0$ for every $u\in V(G)\setminus S$. Notice that for every $u\in V(G)$, $g_u^-$ is an IDF on $H_u-\{u\}$. Moreover, since $S$ is a dominating set of $G$ and   for every  $u\in V(G)\setminus S$ there exists a vertex $y\in N(u)\cap V(H_u)$ with $g(y)>0$, we conclude that  $g$ is an IDF on $G\circ_v H$ of weight $n(G)(\gamma_{_I}(H)-1)+\gamma(G)$, concluding that $\gamma_{_I}(G\circ_v H)\leq n(G)(\gamma_{_I}(H)-1)+\gamma(G)$. To show that in fact this is  an equality, we observe that Lemma~\ref{CorollaryA_fDominante} and Lemma  \ref{lemmaA} lead to
\begin{align*}
\gamma_{_I}(G\circ_v H)&\geq \vert \mathcal{A}_f\vert \gamma_{_I}(H)+\vert\mathcal{B}_f\vert(\gamma_{_I}(H)-1)\\
&=n(G)(\gamma_{_I}(H)-1)+\vert \mathcal{A}_f\vert \\
&\ge n(G)(\gamma_{_I}(H)-1)+\gamma(G).
\end{align*}
Hence, $\gamma_{_I}(G\circ_v H)= n(G)(\gamma_{_I}(H)-1)+\gamma(G)$.

From now on we suppose that $N(x)\cap V(H_x)\subseteq V_0$ for every $x\in  \mathcal{B}_f$.
 Notice that in this case,  $|N(x)\cap \mathcal{A}_f \cap  V_1 |\ge 2$ or $|N(x)\cap \mathcal{A}_f \cap  V_2)|\ge 1$ for  every vertex $x\in \mathcal{B}_f$. Furthermore, since $f$ satisfies Lemma \ref{existenciaf},  $\mathcal{A}_f\subseteq V_1\cup V_2$. Hence, the restriction of $f$ to $V(G)$ is an IDF on $G$, and so  $$\sum_{x\in \mathcal{A}_f}f(x)\ge \gamma_{_I}(G).$$
 Since  $f$  satisfies  Lemma \ref{existenciaf},   we can differentiate the following cases.
 
 \vspace{0,3cm}
\noindent Case 1. $\mathcal{A}_f=\mathcal{A}_f^{1,m}$.
In this case,  
\begin{align*}
\gamma_{_I}(G\circ_v H)&=|\mathcal{A}_f|\gamma_{_I}(H)+|\mathcal{B}_f|(\gamma_{_I}(H)-1)\\
&=n(G)(\gamma_{_I}(H)-1)+|\mathcal{A}_f|\\
&\ge n(G)(\gamma_{_I}(H)-1)+\gamma_{_I}(G).
\end{align*}
Hence, by Proposition \ref{lemma-bound-B} we conclude that $\gamma_{_I}(G\circ_v H)=n(G)(\gamma_{_I}(H)-1)+\gamma_{_I}(G)$.

 \vspace{0,3cm}
\noindent Case 2. $\mathcal{A}_f=\mathcal{A}_f^{2,m}$.  By Lemma~\ref{CorollaryA_fDominante} we have that $|\mathcal{A}_f|\ge \gamma(G)$, so that 
\begin{align*}
\gamma_{_I}(G\circ_v H)&=|\mathcal{A}_f|\gamma_{_I}(H)+|\mathcal{B}_f|(\gamma_{_I}(H)-1)\\
&=n(G)(\gamma_{_I}(H)-1)+|\mathcal{A}_f|\\
&\ge n(G)(\gamma_{_I}(H)-1)+\gamma(G).
\end{align*}
 To show the equality, we take  a $\gamma(G)$-set $S$ and fix $x\in \mathcal{A}_f$ and $y\in \mathcal{B}_f$. Consider the function $g$ on $G\circ_v H$ such that for every $u\in S$, $g_u$ is induced by $f_x$ and for every $u\in V(G)\setminus S$, $g_u$ is induced by $f_y$. Then, $g(u)=2$
 for every $u\in S$ and we have that  $g$ is an IDF on $G\circ_v H$ of weight $n(G)(\gamma_{_I}(H)-1)+\gamma(G)$, concluding that $\gamma_{_I}(G\circ_v H)= n(G)(\gamma_{_I}(H)-1)+\gamma_{}(G)$.

 \vspace{0,3cm}
\noindent Case 3. $\mathcal{A}_f=\mathcal{A}_f^{2,m+1}$. By Lemma \ref{CorollaryA_fDominante} we have that $|\mathcal{A}_f|\ge \gamma(G)$ and since $\gamma_{_I}(G)\le 2\gamma(G)$ we deduce that
 \begin{align*}
 \gamma_{_I}(G\circ_v H)&=|\mathcal{A}_f|(\gamma_{_I}(H)+1)+|\mathcal{B}_f|(\gamma_{_I}(H)-1)\\
 &=n(G)(\gamma_{_I}(H)-1)+2|\mathcal{A}_f|\\
 &\ge n(G)(\gamma_{_I}(H)-1)+2\gamma(G)\\
 &\ge n(G)(\gamma_{_I}(H)-1)+\gamma_{_I}(G).
 \end{align*} Hence,  by Proposition \ref{lemma-bound-B} we conclude that  $\gamma_{_I}(G\circ_v H)=n(G)(\gamma_{_I}(H)-1)+\gamma_{_I}(G)$.

 \vspace{0,3cm}
\noindent Case 4. $\mathcal{A}_f=\mathcal{A}_f^{1,m}\cup\mathcal{A}_f^{2,m+1}$. In this case,  $$|\mathcal{A}_f^{1,m}|+2|\mathcal{A}_f^{2,m+1}|=\sum_{x\in \mathcal{A}_f}f(x)\ge \gamma_{_I}(G).$$ Thus,
 \begin{align*}
\gamma_{_I}(G\circ_v H)&=|\mathcal{A}_f^{1,m}|\gamma_{_I}(H)+|\mathcal{A}_f^{2,m+1}|(\gamma_{_I}(H)+1)+|\mathcal{B}_f|(\gamma_{_I}(H)-1)\\&=n(G)(\gamma_{_I}(H)-1)+|\mathcal{A}_f^{1,m}|+2|\mathcal{A}_f^{2,m+1}|\\
&\ge n(G)(\gamma_{_I}(H)-1)+\gamma_{_I}(G).
\end{align*} 
Finally,  by Proposition \ref{lemma-bound-B} we conclude that  $\gamma_{_I}(G\circ_v H)=n(G)(\gamma_{_I}(H)-1)+\gamma_{_I}(G)$.

Therefore, $\gamma_{_I}(G\circ_v H)\in \{n(G)(\gamma_{_I}(H)-1)+\gamma(G), n(G)(\gamma_{_I}(H)-1)+\gamma_{_I}(G), n(G)\gamma_{_I}(H)\}.$
The remaining statements follow from the previous analysis.
\end{proof}

\begin{corollary} \label{The cae H=K_2}
For any graph $G$ and $v\in V(K_2)$, $$\gamma_{_I}(G\circ_v K_2)=n(G)+\gamma(G).$$
\end{corollary}

\begin{proof}
By Theorem \ref{teo-bounds-rooted} we have that $\gamma_{_I}(G\circ_v K_2)\ge n(G)+\gamma(G).$ To conclude the proof we only need to observe that from any $\gamma(G)$-set $D$ we can define an IDF $f(W_0,W_1,W_2)$ on $G\circ_vK_2$ in such a way that $W_0=V(G)\setminus D$ and $W_2=\emptyset$. Since $\gamma_{_I}(G\circ_v K_2)\le \omega(f)= n(G)+\gamma(G)$, the result follows.
\end{proof}

\begin{corollary}
Let  $G$ and $H$ be two graphs and let $v\in V(H)$. If $n(H)\ge 3$, then $\gamma_{_I}(G\circ_v H)\ge 2n(G)$ and the equality holds if and only if $\gamma_{_I}(H)=2.$
\end{corollary}
\begin{proof}
Let $f$ be a $\gamma_{_I}(G\circ_v H)$-function. By Theorem \ref{teo-bounds-rooted} we differentiate two cases. First, if $\gamma_{_I}(G\circ_v H)= n(G)\gamma_{_I}(H)$, then we immediately conclude that $\gamma_{_I}(G\circ_v H)\ge  2 n(G)$ and the equality holds if and only if $\gamma_{_I}(H)=2.$

Now, if $\gamma_{_I}(G\circ_v H)\ne n(G)\gamma_{_I}(H)$, then 
Theorem \ref{teo-bounds-rooted} and Corollary \ref{Corollary-of-Lemma1} lead to $\gamma_{_I}(G\circ_v H)\ge n(G)(\gamma_{_I}(H)-1)+\gamma(G)\ge 2n(G)+\gamma(G)>2n(G)$.
Therefore, the result follows.
\end{proof}

From the results above we can summarize the case where $\gamma_{_I}(H)=2$ as follows. 

\begin{theorem}\label{RemarkCase=2paraH}
Let $G$ and $H$ be two graphs. If $\gamma_{_I}(H)=2$, then 
$$\gamma_{_I}(G\circ_v H)= \left \{ \begin{array}{lll}
 n(G)+\gamma(G), &  \text{if and only if }H\cong K_2 ;\\
 \\
n(G)+\gamma_{_I}(G), &  \text{if and only if }H\cong \overline{K}_2;
\\
\\ 2n(G), &  \text{otherwise.}
\end{array} \right.$$
\end{theorem}


From now on, the graph obtained from $H$ by removing vertex $v$ will be denoted by $H-\{v\}$. Notice that any $\gamma_{_I}(H-\{v\})$-function  can be extended to an IDF on $H$ by assigning the value $1$ to $v$, which implies that the following lemma holds. 

\begin{lemma}\label{m-1}
For any nontrivial graph $H$  and any $v\in V(H)$,
$$\gamma_{_I}(H-\{v\})\ge \gamma_{_I}(H)-1.$$
\end{lemma}

In order to stablish a sufficient and necessary condition to assure that $\gamma_{_I}(G\circ_v H)=n(G)\gamma_{_I}(H)$ when $\gamma_{_I}(G)<n(G)$, we need to state the following lemma.

\begin{lemma}\label{LemmaBf-empty}Let $f$ be a $\gamma_{_I}(G\circ_v H)$-function. If $\mathcal{B}_f\ne \emptyset$, then $\gamma_{_I}(H-\{v\})= \gamma_{_I}(H)-1$. 
\end{lemma}

\begin{proof}
If there exists $x\in \mathcal{B}_f$, then $\omega(f_x)=\gamma_{_I}(H)-1$ and  $f(x)=0$ (by Lemma \ref{restriction-H}), which implies  that $f_x^-$ is an IDF on $H_x-\{x\}$ of weight $\gamma_{_I}(H)-1$, and so  $\gamma_{_I}(H-\{v\})\le \gamma_{_I}(H)-1$. By Lemma \ref{m-1} we conclude the proof. 
\end{proof}

The following result is straightforward.
\begin{remark}\label{Igualdad-al-orden}
  $\gamma_{_I}(G)=n(G)$ if and only if $G$ has maximum degree $\delta_{\max}(G)\le 1$.
  \end{remark}

\begin{theorem}\label{characterization1}
Let $G$ be a  graph of maximum degree $\delta_{\max}(G)\ge 2$. Given a graph $H$ and a vertex $v\in V(H)$, $\gamma_{_I}(G\circ_v H)=n(G)\gamma_{_I}(H)$ if and only if $\gamma_{_I}(H-\{v\})\ge \gamma_{_I}(H)$.
\end{theorem}
\begin{proof} Suppose  that $\gamma_{_I}(H-\{v\})< \gamma_{_I}(H)$. In such a case,  $\gamma_{_I}(H-\{v\})=\gamma_{_I}(H)-1$ by Lemma \ref{m-1}. Hence,  from any $\gamma_{_I}(H-\{v\})$-function and any $\gamma_{_I}(G)$-function we can construct an IDF on $G\circ_v H$ of weight $n(G)(\gamma_{_I}(H)-1)+\gamma_{_I}(G)$, which implies  that $ \gamma_{_I}(G\circ_v H)\le  n(G)(\gamma_{_I}(H)-1)+\gamma_{_I}(G)$, and by Remark \ref{Igualdad-al-orden} we deduce that  $ \gamma_{_I}(G\circ_v H)<n(G)\gamma_{_I}(H)$. Therefore, if $\gamma_{_I}(G\circ_v H)=n(G)\gamma_{_I}(H)$, then  $\gamma_{_I}(H-\{v\})\ge \gamma_{_I}(H)$.

Now, assume that $\gamma_{_I}(H-\{v\})\ge \gamma_{_I}(H)$ and let $f$ be  a $\gamma_{_I}(G\circ_v H)$-function. By Lemma \ref{LemmaBf-empty} we have that    $\mathcal{B}_f=\emptyset$, and so Theorem \ref{teo-bounds-rooted} leads to
 $\gamma_{_I}(G\circ_v H)=n(G)\gamma_{_I}(H)$.
\end{proof}

It was shown in  \cite{CHELLALI201622}  that $\gamma_{_I}(C_t)=\left\lceil \frac{t}{2}\right\rceil$ for every $t\geq 3$ and $\gamma_{_I}(P_t)=\left\lceil \frac{t+1}{2}\right\rceil$ for every $t\geq 1$. Since $\gamma_{_I}(C_{t}-\{v\})=\gamma_{_I}(P_{t-1})=\left\lceil \frac{t}{2}\right\rceil=\gamma_{_I}(C_{t})$ for every $t\geq 3$,   Theorem \ref{characterization1} leads to the following result.

\begin{corollary}
If $G$ be a graph, $v\in V(C_t)$ and $t\geq 3$, then
$$\gamma_{_I}(G\circ_v C_t)= 
 n(G)\left\lceil \frac{t}{2}\right\rceil .$$
\end{corollary}

From Lemma \ref{m-1} and Theorems \ref{teo-bounds-rooted} and \ref{characterization1} we deduce the following result. 

\begin{theorem}\label{Characterization1.1}
Let $G$ be a graph of maximum degree $\delta_{\max}(G)\ge 2$.  Given a graph $H$ and a  vertex $v\in V(H)$,  the following statements are equivalent. 
\begin{itemize}
\item $\gamma_{_I}(G\circ_v H)= n(G)(\gamma_{_I}(H)-1)+\gamma_{_I}(G)$ or $\gamma_{_I}(G\circ_v H)= n(G)(\gamma_{_I}(H)-1)+\gamma(G)$.
\item $\gamma_{_I}(H-\{v\})=\gamma_{_I}(H)-1$.
\end{itemize}
\end{theorem}

We now focus on the case of graphs $G$ with 
 $\gamma_{_I}(G)>\gamma(G)$. 

\begin{theorem}\label{characterization2}
Let $G$ be a graph  of maximum degree  $\delta_{\max}(G)\ge 2$  with  $ \gamma_{_I}(G)>\gamma(G)$.  For any graph $H$ and any vertex $v\in V(H)$, $\gamma_{_I}(G\circ_v H)=n(G)(\gamma_{_I}(H)-1)+\gamma(G)$ if and only if  $\gamma_{_I}(H-\{v\})=\gamma_{_I}(H)-1$ and one of the following conditions holds.
\begin{enumerate}[{\rm (i)}]
 \item   There exists a $\gamma_{_I}(H-\{v\})$-function $g$ such that $g(y)>0$ for some $y\in N(v)$.
 \item  There exists a $\gamma_{_I}(H)$-function $h$ such that $h(v)=2$.
\end{enumerate}
\end{theorem}

\begin{proof}
Assume that $\gamma_{_I}(G\circ_v H)=n(G)(\gamma_{_I}(H)-1)+\gamma(G)$. By Theorem \ref{Characterization1.1}, $\gamma_{_I}(H-\{v\})=\gamma_{_I}(H)-1$. Suppose by contradiction that conditions (i) and (ii) do not hold.
Let $f$ be a $\gamma_{_I}(G\circ_v H)$-function. Since $\gamma(G)<\gamma_{_I}(G)< n(G)$, we have that $\gamma_{_I}(G\circ_v H)=n(G)(\gamma_{_I}(H)-1)+\gamma(G)<n(G)\gamma_{_I}(H)$,  concluding that $\mathcal{B}_f\neq \emptyset$ by Theorem \ref{teo-bounds-rooted}. We can assume that $f$ satisfies  Lemma \ref{existenciaf} and so $\mathcal{A}_f\in \{\mathcal{A}_f^{1,m},\mathcal{A}_f^{2,m},\mathcal{A}_f^{2,m+1},\mathcal{A}_f^{1,m}\cup\mathcal{A}_f^{2,m+1}\}$. Moreover, $\mathcal{A}_f\neq \mathcal{A}_f^{2,m}$ since  (ii) does not hold. 
 For any $x\in \mathcal{B}_f$, we have that $f(x)=0$ (by Lemma \ref{restriction-H}), which implies that $f_x^-$ is $\gamma(H-\{x\})-$function, and since (i) does not hold, $N(x)\cap V(H_x)\subseteq V_0$. Hence, we only have to consider Cases 1, 3 and 4 of the proof of Theorem \ref{teo-bounds-rooted}, to obtain that $\gamma_{_I}(G\circ_v H)=n(G)(\gamma_{_I}(H)-1)+\gamma_{_I}(G)$, which is a contradiction as $\gamma(G)<\gamma_{_I}(G)$. Hence, conditions  (i) and (ii) hold.

Now, assume that $\gamma_{_I}(H-\{v\})=\gamma_{_I}(H)-1$. First, suppose that condition (i) holds. So, consider a $\gamma_{_I}(H-\{v\})$-function $h$ such that $h(y)>0$ for some $y\in N(v)$.   Let $S$ be a $\gamma(G)$-set and consider the function $l$ on $G\circ_v H$ such that for every vertex $x\in V(G)$, $l^-_x$ is induced by $h$, $l(x)=1$ if $x\in S$ and $l(x)=0$ if $x\not\in S$. Notice that $l$ is an IDF on $G\circ_v H$ of weight $\omega(l)=n(G)(\gamma_{_I}(H)-1)+\gamma(G)$, which implies that $\gamma_{_I}(G\circ_v H)\leq n(G)(\gamma_{_I}(H)-1)+\gamma(G)$. Thus,  by Theorem \ref{Characterization1.1} we conclude  that $\gamma_{_I}(G\circ_v H)=n(G)(\gamma_{_I}(H)-1)+\gamma(G)$.
Now, suppose that (i) does not hold and (ii) holds. As $\gamma_{_I}(H-\{v\})=\gamma_{_I}(H)-1$ and  $\delta_{\max}(G)\ge 2$, by  Theorem \ref{Characterization1.1} we have  that $\gamma_{_I}(G\circ_v H)<n(G)\gamma_{_I}(H)$.  Hence, by Theorem \ref{teo-bounds-rooted} we conclude that  $\mathcal{B}_g\neq \emptyset$ for every $\gamma_{_I}(G\circ_v H)$-function $g$. We can assume that $g$ satisfies Lemma \ref{existenciaf}, i.e.,   $\mathcal{A}_g\in \{\mathcal{A}_g^{1,m},\mathcal{A}_g^{2,m},\mathcal{A}_g^{2,m+1},\mathcal{A}_g^{1,m}\cup\mathcal{A}_g^{2,m+1}\}$. Moreover,  since condition (ii) holds, we can claim that $\mathcal{A}_g^{2,m}\ne\emptyset$, so that  $\mathcal{A}_g=\mathcal{A}_g^{2,m}$. 
Now, for any $x\in \mathcal{B}_g$, we have that $g(x)=0$ and $g_x^-$ is $\gamma(H-\{x\})$-function and, since (i) does not hold, $N(x)\cap V(H_x)\subseteq V_0$. To conclude the proof we only have to consider  Case 2 of the proof of Theorem \ref{teo-bounds-rooted}, obtaining that $\gamma_{_I}(G\circ_v H)=n(G)(\gamma_{_I}(H)-1)+\gamma(G)$.
\end{proof}

From Theorems \ref{characterization1} and \ref{characterization2} we deduce the following result. 

\begin{theorem}\label{rooted-paths}
Let $G$ be a graph and $t\geq 2$. If $v\in L(P_t)$, then
$$\displaystyle\gamma_{_I}(G\circ_v P_t)= \left \{ \begin{array}{lll}
 n(G)\left\lceil \frac{t+1}{2}\right\rceil, &  t\equiv 1  \pmod 2 ;\\
 \\
n(G) \left\lceil \frac{t}{2}\right\rceil +\gamma(G), &  t\equiv 0  \pmod 2 .
\end{array} \right.$$
Furthermore, if $v\in V(P_t)\setminus L(P_t)$, then
$$\gamma_{_I}(G\circ_v P_t)=n(G)\left\lceil \frac{t+1}{2}\right\rceil$$
\end{theorem}
\begin{proof}
The case $t\equiv 1  \pmod 2$ for any $v$ is deduced from Theorem \ref{characterization1}, while the case $t\equiv 0  \pmod 2$ for $v\in V(P_t)\setminus L(P_t)$ is deduced from Theorem \ref{characterization2}. 
\end{proof}

From Theorems \ref{Characterization1.1} and \ref{characterization2} we inmediately have the following result. 

\begin{theorem}\label{characterization3}
 Let $G$ be a graph of maximum degree  $\delta_{\max}(G)\ge 2$  with $\gamma(G)<\gamma_{_I}(G)$. For any graph $H$ and any vertex $v\in V(H)$,  $\gamma_{_I}(G\circ_v H)=n(G)(\gamma_{_I}(H)-1)+\gamma_{_I}(G)$ if and only if  $\gamma_{_I}(H-\{v\})=\gamma_{_I}(H)-1$ and the following conditions hold:
\begin{itemize}
 \item [(i)] For every $\gamma_{_I}(H-\{v\})$-function $g$, $g(y)=0$ for every $y\in N(v)$.
 \item [(ii)] For every $\gamma_{_I}(H)$-function $h$, $h(v)\neq 2$.
\end{itemize}
\end{theorem}


\begin{theorem}
Let $G$ be a graph with $\delta_{\max}(G)\ge 2$,   $H$ a  graph and $u\in V(H)$. If $f(u)=2$ for every $\gamma_{_I}(H)$-function $f$,  then for every $v\in N(u)$, $$\gamma_{_I}(G\circ_v H)=n(G)\gamma_{_I}(H).$$
\end{theorem}

\begin{proof}
Assume that $f(u)=2$ for every $\gamma_{_I}(H)$-function $f$, and let $v\in N(u)$. Suppose  to the contrary that $\gamma_{_I}(G\circ_v H)\neq n(G)\gamma_{_I}(H)$. In such a case, by Theorem \ref{characterization1} and Lemma~\ref{m-1} 
we conclude that $\gamma_{_I}(H-\{v\})=\gamma_{_I}(H)-1$.   Let $g$ be a $\gamma_{_I}(H-\{v\})$-function. If $g(u)=2$, then we define a  function $h$ on $H$ such that $h(w)=g(w)$ for every $w\neq v$ and $h(v)=0$. 
Observe that $h$ is an IDF on $H$  with $\omega(h)=\omega(g)=\gamma_{_I}(H)-1$, which is a contradiction. If $g(u)\leq 1$, then we define a function $h$ on $H$ such that $h(w)=g(w)$ if $w\neq v$ and $h(v)=1$. In this case,  $h$ is a $\gamma_{_I}(H)$-function with $h(u)\neq 2$, which is a contradiction. Therefore, $\gamma_{_I}(G\circ_v H)=n(G)\gamma_{_I}(H)$.
\end{proof}

The next theorem considers the case in which the root of $H$ is a strong support vertex. 
A leaf of a graph $H$ is a vertex of degree one while a strong support vertex of $H$ is a vertex adjacent to at least two leaves. We denote the set of leaves of $H$ as $L(H)$ and the set of strong support vertices of $H$ as $S(H)$.

\begin{theorem}
Let $G$ and $H$ be two  graphs. If $v\in S(H)$ then $$\gamma_{_I}(G\circ_v H)=n(G)\gamma_{_I}(H).$$
\end{theorem}
\begin{proof}
By Theorem \ref{characterization1}, it is enough to show that $\gamma_{_I}(H-\{v\})\geq\gamma_{_I}(H)$. Notice that for any $\gamma_{_I}(H-\{v\})$-function $g$ and any $u\in L(H)\cap N(v)$  we have that $g(u)=1$. Since $|N(v)\cap L(H)|\ge 2$, the function $f$ defined  on $H$ as $f(v)=0$ and $f(w)=g(w)$ if $w\in V(H)-\{v\}$ is an IDF on $H$ concluding that $\gamma_{_I}(H-\{v\})\geq \omega(g)=\omega(f)=\gamma_{_I}(H)$, as required.
\end{proof}

\begin{theorem}\label{forcedf}
Let $G$ be a graph with $\delta_{\max}(G)\ge 2$, $H$ a  graph and $v\in V(H)$.  If $g(v)\neq 1$ for every $\gamma_{_I}(H)$-function $g$,  then $$\gamma_{_I}(G\circ_v H)=n(G)\gamma_{_I}(H).$$
\end{theorem}

\begin{proof}
Assume that $g(v)\neq 1$ for every $\gamma_{_I}(H)$-function $g$, and suppose that $\gamma_{_I}(G\circ_v H)\neq n(G)\gamma_{_I}(H)$. By Lemma \ref{m-1} and  Theorem \ref{characterization1} we have that $\gamma_{_I}(H-\{v\})=\gamma_{_I}(H)-1$. Let $f$ be a $\gamma_{_I}(H-\{v\})$-function and consider the function $h$ on $H$ such that $h(v)=1$ and $h(u)=f(u)$ for every $u\ne v$. Notice that $h$ is a $\gamma_{_I}(H)$-function on $H$ with $h(v)=1$, which is a contradiction. Therefore, $\gamma_{_I}(G\circ_v H)=n(G)\gamma_{_I}(H)$.
\end{proof}

\section{The case of corona graphs}

Given two graphs $G$  and $H$, the corona product $G\odot H$ is defined as the graph obtained from $G$ and $H$ by taking one copy of $G$ and $n(G)$ copies of $H$ and joining by an edge each vertex of the $i^{th}$ copy of $H$ with the  $i^{th}$ vertex of $G$  for each $i\in \{1,\dots,n(G)\}$.

The join $G + H$ is defined as the graph obtained from disjoint graphs $G$ and $H$ by taking
one copy of $G$ and one copy of $H$ and joining by an edge each vertex of $G$ with each vertex of $H$. Notice that the corona product graph
$K_1\odot H$ is isomorphic to the join graph $K_1 + H$. Furthermore, any corona product graph 
 $G\odot H$ can be seen as a rooted product, i.e.,  $$G\odot H\cong G\circ_v(K_1+H),$$ where $v$ is the vertex of $K_1$.
 Since $\gamma_{_I}(K_1+ H)=2$, by Theorem \ref{RemarkCase=2paraH} we deduce the following result.
 
\begin{corollary}
For any graph  $G$ and any graph $H$,
 $$\gamma_{_I}(G\odot H)= \left \{ \begin{array}{ll}  n(G)+\gamma(G), & \text{ if }  H\cong K_1;\\
\\ 2n(G), & \text{otherwise.}
\end{array}\right.$$
\end{corollary}

\section{NP-Hardness}

Given a positive integer $k$ and a graph $G$, the  problem of deciding if  $G$ has an Italian dominating  function   $f$ of weight  $\omega(f)\le k$  is NP-complete \cite{CHELLALI201622}. Therefore, the  problem of computing the Italian domination number of a graph is NP-hard. In this section we will show an alternative way of reaching this conclusion.

Recently some authors have  shown how graph products can become  useful tools to show that some optimization problems are NP-hard.  For instance,  Fernau and Rodr\'{i}guez-Vel\'{a}zquez \cite{RV-F-2013,MR3218546}  have shown that the corona product of two graphs can be used to derive NP-hardness results on the (local) metric dimension, based on known NP-hardness results on the (local) adjacency dimension. In the same direction,  Dettlaff et al.\ \cite{Dettlaff-LemanskaRodrZuazua2017} have shown how we can use the   lexicographic product of two graphs  to deduce an NP-hardness result on the super domination number, from a well-known NP-hardness result on the independence number of a graph. 
In Theorem \ref{NP-hardTh} we will show that we can use the rooted product of two graphs to study the computational complexity of the problem of finding the Italian domination number of a graph. 
In this case, we will use Corollary \ref{The cae H=K_2} and the fact  that the problem of computing the domination number of a graph is NP-hard., i.e., given a positive integer $k$ and a graph $G$, the problem of deciding if $G$ has a dominating set $D$ of cardinality $|D|\le k$  is NP-complete \cite{Garey1979}, which implies that the optimization problem of computing the domination number of a graph is NP-hard.

\begin{theorem}\label{NP-hardTh}
The problem of computing the Italian domination number of a graph is NP-hard.
\end{theorem}

\begin{proof}
 By Corollary \ref{The cae H=K_2}, for any graph $G$   we have that $$\gamma_{_I}(G\circ_v K_2)=n(G)+\gamma(G),$$ where $v$ is a leaf of $K_2$. Hence, the problem of computing $\gamma(G)$ is equivalent to the problem of finding $\gamma_{_I}(G\circ_v K_2)$, which implies that 
 the  problem of computing the Italian domination number of a graph is NP-hard.
\end{proof}

\end{document}